\documentclass[10pt]{amsart}
\usepackage[totalwidth=480pt, totalheight=640pt]{geometry}
\usepackage{amsmath, amssymb, amsthm, amsfonts, mathrsfs, eucal, enumerate, rotating, lscape,layout}
\usepackage[all,tips,2cell]{xy}
\usepackage[usenames]{color}
\usepackage{xspace}
\usepackage{soul}
\usepackage{ifpdf}
\usepackage{ifthen}
\usepackage{hyperref}
\newtheorem{theorem}{Theorem}[section]
\newtheorem*{theorem*}{Theorem}
\newtheorem{corollary}[theorem]{Corollary}
\newtheorem*{corollary*}{Corollary}

\newtheorem*{claim*}{Claim}
\newtheorem*{lemma*}{Lemma}
\newtheorem{proposition}[theorem]{Proposition}
\newtheorem*{proposition*}{Proposition}
\theoremstyle{remark}
\newtheorem{remark}[theorem]{Remark}

\newtheorem{example}[theorem]{Example}
\numberwithin{equation}{section}

\newcommand{\id}{\mathrm{id}}
\newcommand{\ass}{\mathsf{a}}

\newcommand{\alman}[1]{\mathfrak{#1}}

\newcommand{\gotikl}{\alman l}
\newcommand{\gotikr}{\alman r}

\newcommand{\invass}{\mathsf{a}^{-1}}
\newcommand{\comm}{\mathsf{c}}
\newcommand{\boyut}[1]{\mathscr{#1}}
\newcommand{\unitstack}{\mathscr{I}}
\newcommand{\twounitstack}{\mathbb{I}}
\newcommand{\unit}{\mathrm{I}}

\newcommand{\twounit}{\mathcal{I}}

\newcommand{\calC}{\mathcal{C}}
\newcommand{\calD}{\mathcal{D}}
\newcommand{\TORS}{\textsc{Tors}}

\DeclareMathOperator{\Hom}{Hom}
\DeclareMathOperator{\rmD}{D}
\DeclareMathOperator{\rmC}{C}

\DeclareMathOperator{\car}{\circlearrowright}
\newcommand{\oneA}{\boyut A}

\newcommand{\oneC}{\boyut C}
\newcommand{\oneD}{\boyut D}
\newcommand{\oneI}{\boyut I}
\newcommand{\oneJ}{\boyut J}

\newcommand{\oneL}{\boyut L}
\newcommand{\oneG}{\boyut G}

\newcommand{\oneM}{\boyut M}
\newcommand{\oneP}{\boyut P}
\newcommand{\oneQ}{\boyut Q}
\newcommand{\ES}{\mathsf{S}}
\newcommand{\twoA}{\mathbb{A}}

\newcommand{\twoC}{\mathbb{C}}
\newcommand{\twoD}{\mathbb{D}}
\newcommand{\twoG}{\mathbb{G}}
\newcommand{\twoM}{\mathbb{M}}
\newcommand{\ra}{{\rightarrow}}
\newcommand{\Ra}{{\Rightarrow}}
\newcommand{\La}{{\Leftarrow}}
\newcommand{\Da}{{\Downarrow}}
\newcommand{\Ua}{{\Uparrow}}

\newcommand{\kesir}[2]{\genfrac{}{}{0pt}{}{#1}{#2}}

\xyoption{2cell}
\xyoption{rotate}

\begin{document}
\UseAllTwocells
\title{Notes on Weak Units of Group-Like 1- and 2-Stacks}
\author{Ettore Aldrovandi}
\author{Ahmet Emin Tatar}
\address{Department of Mathematics, FSU, Tallahassee, USA}
\email{aldrovandi@math.fsu.edu}
\address{Department of Mathematics and Statistics, KFUPM, Dhahran, KSA}
\email{atatar@kfupm.edu.sa}
\date{}
\maketitle
\begin{abstract}
The weak units of strict monoidal 1- and 2-categories are defined respectively in \cite{MR2388233} and \cite{MR3035770}. In this paper, we define them for group-like 1- and 2-stacks. We show that they form a contractible Picard 1- and 2-stack, respectively. We give their cohomological description which provides for these stacks a representation by complexes of sheaves of groups. Later, we extend the discussion to the monoidal case. We consider the (2-)substack of cancelable objects of a monoidal 1-(2-)stack. We observe that this (2-)substack is trivially group-like, its weak units are same as the weak units of the monoidal 1-(2-)stack, and therefore we can recover the contractibility results in \cite{MR2388233} and \cite{MR3035770} by analyzing it.
\end{abstract}

\section{Introduction}
Saavedra in \cite{MR0338002} gives an alternative way of defining units in monoidal categories. He observes that a unit $e$ in monoidal category $\rmC$ can be defined as a cancellable idempotent object, an object $e$ with the property that tensoring with $e$ from both sides is an equivalence and that is equipped with isomorphism $\varphi:ee \ra e$. In the traditional way, a unit is an object equipped with left and right constraints (i.e. the isomorphisms $\gotikl_X:eX \ra X$ and $\gotikr_X:Xe \ra X$) satisfying some compatibility conditions. In \cite{MR2388233}, Kock analyzes these two definitions of units in a monoidal category. He calls the units defined as cancellable idempotent objects \emph{Saavedra units}, and the units extracted from the definition of bicategories with one object \emph{classical units}. He shows that these two notions of units are equivalent and the category they form is contractible.

In a subsequent work \cite{MR3035770}, Joyal and Kock carry out the discussion for units of monoidal categories to units of monoidal 2-categories. They give an alternative definition to the notion of classical unit in monoidal 2-categories. In this classical notion a unit is an object equipped with left and right constraints which are weakly invertible 1-morphisms and with a 2-isomorphism between the left and the right constraints. These data are required to satisfy certain conditions (see \cite[\S 6]{MR3035770}). On the other hand, Joyal and Kock define a unit of a monoidal 2-category as an appropriate generalization of Saavedra unit. This is an object $e$ with the property that tensoring with $e$ from both sides is biequivalence and that is equipped with the weakly invertible morphism $\varphi:ee \ra e$. Throughout this paper, we call this alternative definition of unit \emph{Joyal-Kock unit}. In \cite{MR3035770}, Joyal and Kock show, as in the 1-categorical case, that these two notions of units are equivalent and their 2-category is contractible.   

The language of Saavedra units and Joyal-Kock units in the context of group-like 1- and 2-stacks is very helpful. Their capability of expressing units without referring to left and right constraints is very beneficial if one considers the amount of data and coherence conditions required to define group-like 1- and 2-stacks. There is no need of mentioning units in their definitions because the fact \textit{every object is cancellable} is part of the group-like data and these notions of units are equivalent to classical notions. The benefits of this economical way of defining group-like 1- and 2-stacks becomes more significant when it comes to define (2-)functors. With Saavedra units and Joyal-Kock units, we don't need to assume that a unit is transferred to a unit. It is enough to assure that the (2-)functor transfers the group-like structure to the group-like structure. Another benefit of Saavedra and Joyal-Kock units can be observed when extensions are described in terms of torsors. In \cite[The\'or\`eme 3.2.2]{MR1191733}, it is shown that given an extension of a discrete group $K$ by a group-like stack $\oneG$ the $\oneG$-bitorsor associated to this extension should be trivial when it is pull backed over the unit element of $K$ so that this trivialization is compatible with other torsor structures. However in a recent work \cite{Bertolin_Tatar}, it is noticed that this trivialization condition becomes redundant if one works with Saavedra units in case of extensions of Picard stacks or with Joyal-Kock units in the case of Picard 2-stacks.

In this paper, we define the category (resp. 2-category) of Saavedra units (resp. Joyal-Kock units) in a group-like stack (resp. group-like 2-stack). We prove they are contractible. To this end we follow a much direct approach. As opposed to showing that the (2)-category of classical units of a monoidal (2)-category is contractible and showing that the (2-)category of Saavedra(Joyal-Kock) units is equivalent to the (2-)category of classical units, as it is done in \cite{MR2388233} and in \cite{MR3035770}, we give explicit construction of the morphisms between Saavedra units and Joyal-Kock units. Due to the contractibility Saavedra units and Joyal-Kock units form a Picard stack (resp. 2-stack) of their own which we denote by $\unitstack(\oneC)$ (resp. $\twounitstack(\twoC)$). We expect them to be equivalence relations with contractible quotients, i.e, to be ultimately contractible spaces. We confirm this both by a direct geometric analysis and by explicitly computing complexes of sheaves of groups that represent them. Although it immediately follows from the contractibility that these complexes should be quasi-isomorphic to the zero complex, their explicit computation is interesting because we compare them directly to the complexes representing the homotopy fiber over 1 in the Postnikov exact sequence. The comparison allows us to characterize the Saavedra units as \textit{rigid} model of the homotopy fibers over 1. We study this in more details in \cite{Aldrovandi_Tatar}.

Lastly, we discuss the weak units of monoidal 1- or 2-stacks. We observe that cancelable objects of a monoidal 1-(2-)stack form a  group-like (2-)substack that can be presented as the homotopy cokernel of a (2-)crossed-module. Directly from the definition of weak units, it follows that the weak units of a monoidal 1-(2-)stack and the weak units of the group-like (2-)substack of cancelable objects coincide. Therefore it suffices to study the group-like (2-)substack of cancelable objects of a monoidal 1-(2-)stack to obtain the contractibility results in \cite{MR2388233} and \cite{MR3035770}. This observation also shows that the weak units of a group-like 1-(2-) stack is the case to study.

\paragraph{\textbf{Organization of the paper.}} In section \ref{section:recall}, we quickly recall Saavedra units and Joyal-Kock units of group-like 1- and 2- categories. In section \ref{section:stack_of_units}, we examine Saavedra units of group-like stacks. We show that Saavedra units of a group-like stack $\oneC$ form a contractible Picard stack $\unitstack(\oneC)$. We give a cocyclic description of Saavedra units from which we deduce a complex of abelian sheaves such that the Picard stack associated to it is equivalent to $\unitstack(\oneC)$. We finish this section by extending the discussion to monoidal stacks. In section \ref{section:2stack_of_units}, we follow the same plan as in section \ref{section:stack_of_units} for Joyal-Kock units.

\paragraph{\textbf{Notation and Conventions.}} We work with strict 2-categories. A 2-groupoid is a 2-category whose 1-morphisms are weakly invertible and 2-morphisms are isomorphisms. A 2-functor is used in the sense of \cite{MR0364245}. For compactness in the diagrams, we denote the tensoring operation in any category by juxtaposition. The usual notation $\otimes$ is used in the names of functors (i.e. $X \otimes -$ denotes the functor tensoring by $X$) and in cases to avoid ambiguities. We use capital roman letters for categories ($\rmC$, $\rmD$, $\hdots$), calligraphic letters for 2-categories ($\calC$, $\calD$, $\hdots$), script letters for stacks ($\oneC$, $\oneD$, $\hdots$) and double letters for 2-stacks ($\twoC$, $\twoD$, $\hdots$).

\section{Quick Recall on Weak Units}\label{section:recall}
In this section, we recall briefly the weak units of group-like 1- and 2- categories. The main references are \cite{MR2388233} and \cite{MR3035770} where these units are defined for strict monoidal 1- and 2- categories. For definitions of group-like 1- and 2-categories, we refer to \cite{MR1301844}.
\subsection{Saavedra Units}

Let $\rmC$ be a group-like category. A pair $(e,\varphi)$ is called a unit element where $e$ is an object and $\varphi: e e \ra e$ is an isomorphism in $\rmC$ . A unit morphism $(e_1,\varphi_1) \ra (e_2,\varphi_2)$ is given by an isomorphism $u:e_1 \ra e_2$ in $\rmC$ such that the diagram
\begin{equation*}\label{diagram:morphism_of_units}
\xymatrix{e_1 e_1 \ar@{}[dr]|{\car}\ar[r]^{u u} \ar[d]_{\varphi_1} & e_2 e_2 \ar[d]^{\varphi_2}\\ e_1 \ar[r]_{u} & e_2}
\end{equation*}
commutes. This defines the groupoid of Saavedra units $\unit(\rmC)$.

In \cite{MR2388233} these units are called Saavedra units since they were first mentioned by Saavedra in \cite{MR0338002}. Since the classical notion of unit extracted from the definition of monoidal category is equivalent to the notion of Saavedra unit and $\unit(\rmC)$ is contractible (\cite[Proposition 2.19]{MR2388233}), $\unit(\rmC)$ is always a Picard category. However, if $\rmC$ is braided, we can define the tensor product of Saavedra units without referring to the contractibility. If $(e_1,\varphi_1)$ and $(e_2,\varphi_2)$ are two Saavedra units, then $(e_1,\varphi_1)\otimes (e_2,\varphi_2):=(e_1e_2,\varphi)$ where $\varphi$ is the composition 
\begin{equation}\label{morphism:additive structure}
\xymatrix{(e_1e_2)(e_1e_2) \ar[r]^{\invass} &((e_1e_2)e_1)e_2 \ar[r]^{\invass \comm \ass} & ((e_1e_1)e_2)e_2 \ar[r]^{\ass} & (e_1e_1)(e_2e_2) \ar[r]^(0.6){\varphi_1\varphi_2} &e_1e_2},
\end{equation}
with $\invass \comm \ass$ given by
\begin{equation}\label{morphism:aca}
\xymatrix@1{((e_1e_2)e_1)e_2 \ar[r]^{\invass} & (e_1(e_2e_1))e_2 \ar[r]^{\comm} & (e_1(e_1e_2))e_2  \ar[r]^{\ass} &((e_1e_1)e_2)e_2}.
\end{equation}
The isomorphisms $\ass$ and $\comm$ represent the associativity and the braiding constraints, respectively. There is a choice involved in the definition of $\varphi$, but any two such choices are connected by a unique isomorphism. This unique isomorphism would be the pasting of the isomorphisms of the braiding. We also note that our definition coincides with the one in \cite[\S2.21]{MR2388233} if one assumes strict associativity and uses the compatibility between the braiding constraint and left and right unit constraints.

In \cite{Deligne}, Deligne points out that a Picard category $\rmC$ always has a Saavedra unit, since tensoring by an object $X$ in $\rmC$ is an equivalence (see the proof of Proposition (\ref{proposition:unit_exists_in_Picard_2_category})).  

\subsection{Joyal-Kock Units} \label{section:units_of_picard_2-categories}

Let $\calC$ be a group-like 2-category. A pair $(e, \varphi)$ is called a unit element in $\calC$ where $e$ is an object and $\varphi: e e \ra e$ is a weakly invertible 1-morphism in $\calC$. A unit 1-morphism $(e_1,\varphi_1) \ra (e_2,\varphi_2)$ is given by a pair $(f,\theta_f)$ where $f:e_1 \ra e_2$ is a weakly invertible 1-morphism and $\theta_f$ is the 2-isomorphism

\begin{equation*}\label{diagram:unit_morphism}
\xymatrix{e_1 e_1 \ar[r]^{f  f} \ar[d]_{\varphi_1} \ar@{}[dr]|{\Ua \theta_f}& e_2 e_2 \ar[d]^{\varphi_2}\\
            e_1 \ar[r]_{f} & e_2}
\end{equation*}

A unit 2-morphism $(f,\theta_f) \Ra (g,\theta_g)$ is given by a 2-isomorphism $\delta: f \Ra g$ in $\calC$ such that
\begin{equation*}\label{diagram:unit_2-morphism}
\xymatrix{e_1 e_1 \ar[r]_{f  f}^{\Ua \delta \delta} \ar[d]_{\varphi_1} \ar@/^0.5cm/[r]^{g  g} \ar@{}[dr]|{\Ua \theta_f} & e_2 e_2 \ar[d]^{\varphi_2}\ar@{}[dr]|{=}&e_1 e_1 \ar[r]^{g g} \ar[d]_{\varphi_1} \ar@{}[dr]|{\Ua \theta_g} & e_2 e_2 \ar[d]^{\varphi_2}\\ 
            e_1 \ar[r]_{f}& e_2 & e_1 \ar[r]^{g}_{\Ua \delta} \ar@/_0.5cm/[r]_{f}& e_2}
\end{equation*}

Unit elements, unit 1-morphisms, and unit 2-morphisms of a group-like 2-category $\calC$ form the 2-groupoid $\twounit(\calC)$ of Joyal-Kock units. We define the tensor product on $\twounit(\calC)$ in the same way as the one on $\unit(\rmC)$. As for Saavedra units, $\twounit(\calC)$ is always a Picard 2-category. However if $\calC$ is group-like, then $\twounit(\calC)$ is not empty. 

\begin{proposition}\label{proposition:unit_exists_in_Picard_2_category}
A group-like 2-category $\calC$ always has a Joyal-Kock unit.
\end{proposition}

This result is not surprising since group-like 2-categories have classical units that are equivalent to Joyal-Kock units \cite[Theorem E]{MR3035770}. We give a proof of this fact without referring to this equivalence. 

\begin{proof}[Proof of Proposition \ref{proposition:unit_exists_in_Picard_2_category}]
For any object $X$ in $\calC$, the 2-functor $- \otimes X$ from $\calC$ to $\calC$ is a biequivalence. Therefore, for any $X \in \calC$ there exists $e_X \in \calC$ with a  1-morphism $f:e_X X \ra X$. $\id_{e_X} \otimes f$ is a 1-morphism in the category $\Hom_{\calC}(e_X (e_X X),e_X X)$. $\invass_{e_X,e_X,X} \circ (\id_{e_X} \otimes f)$ is a 1-morphism in the category $\Hom_{\calC}((e_X e_X) X, e_X X)$ which is equivalent to $\Hom_{\calC}(e_X e_X,e_X)$, since tensoring is a biequivalence. We define $\varphi: e_X e_X \ra e_X$ as the image of $\id_{e_X} \otimes f$ under this equivalence. 
\end{proof}

\section{Weak Units of Group-Like Stacks}\label{section:stack_of_units}
We define weak units of a group-like stack $\oneC$. We call these units Saavedra units of $\oneC$. We show that there exists unique isomorphism between the Saavedra units thereby they form a Picard stack $\unitstack(\oneC)$. We give cohomological description of Saavedra units from which we obtain a complex of abelian sheaves that represents $\unitstack(\oneC)$. We end this section by extending the discussion to monoidal case.
\subsection{Saavedra Units of a Group-Like Stack}\label{section:saavedra_units_in_picard_stack}
We consider a group-like stack $\oneG$. By \cite[Theorem 5.3.6]{Aldrovandi2009687}, we assume that $\oneG$ is represented by a  crossed-module $\lambda:G \ra H$ of sheaves of groups. Hence, $\oneG$ can be modeled by $\TORS(G,H)$, the group-like stack of $(G,H)$-torsors. For details of $\TORS(G,H)$, we refer to \cite{MR1086889} and \cite{MR546620}. Here, we give a brief reminder. A $(G,H)$-torsor is a pair $(L,x)$, where $L$ is an $G$-torsor and $x:L \ra H$ is a $G$-equivariant morphism of sheaves. A morphism between two pairs $(L,x)$ and $(K,y)$ is a $G$-equivariant morphism of sheaves $\psi: L \ra K$ such that the diagram

\begin{equation*}\label{diagram:commutative_morphism_1}
\xymatrix@1{L  \ar[dr]_x \ar[rr]^{\psi} & \ar@{}[d]|(0.4){\car}& K  \ar[dl]^y\\
            &H&}
\end{equation*}
commutes. The tensor product on $\TORS(G,H)$ is
\begin{equation}\label{tensor_product}
(L,x) \otimes (K,y):=(L \wedge^{G} K, x \wedge y),
\end{equation}
where $L \wedge^{G} K$ is the contracted product and $x \wedge y$ is the $G$-equivariant morphism from $L \wedge^{G} K$ to $H$ given by $x(l)y(k)$ with $(l,k)$ in $L \wedge^{G} K$.

A \emph{Saavedra unit} in $\TORS(G,H)$ is an idempotent $(G,H)$-torsor, that is, a $(G,H)$-torsor $(I,x)$ with a $(G,H)$-torsor morphism
\begin{equation*}
\varphi:\xymatrix@1{(I,x) \otimes (I,x) \ar[rr] && (I,x).}
\end{equation*}
In other words, we have a $(G,H)$-torsor morphism
\begin{equation*}\label{idempotent}
\varphi:\xymatrix@1{(I \wedge^G I,x^2) \ar[rr] && (I,x).}
\end{equation*}
We denote these units  by $((I,x),\varphi)$.

A \emph{morphism of Saavedra units} in $\TORS(G,H)$
\[\xymatrix@1{((I,x),\varphi) \ar[rr] && ((J,y),\sigma)},\]
is given by a $(G,H)$-torsor morphism $\psi:(I,x) \ra (J,y)$ satisfying the commutative diagram
\begin{equation*}\label{diagram:morphism_of_saavedra_units}
\xymatrix{(I \wedge^G I,x^2) \ar@{}[dr]|{\car} \ar[r]^(0.48){\psi \wedge \psi} \ar[d]_{\varphi} & (J \wedge^G J,y^2) \ar[d]^{\sigma}\\(I,x) \ar[r]_{\psi} & (J,y)}
\end{equation*}
This defines the groupoid of Saavedra units of a group-like stack $\oneG$. We denote it by $\unitstack(\oneG)$. 

\begin{example}\label{kernel_saavedra_unit}
The trivial $(G,H)$-torsor $(G,1)$ is a Saavedra unit and so is any $(G,H)$-torsor globally isomorphic to it.
\end{example}

\subsection{Contractibility of Saavedra Units}\label{section:contractibility_of_saavedra_units} In this section, we prove that the category of Saavedra units of a group-like stack is contractible by directly constructing the unique isomorphism between any two Saavedra units.
\begin{proposition}\label{the_proposition}
All Saavedra units of $\oneG$ are uniquely isomorphic to each other. That is, $\unitstack(\oneG)$ is a contractible groupoid over the site $\ES$.
\end{proposition}

\begin{proof}
Let $((I,x),\varphi)$ be a Saavedra unit of $\oneG$. We note that $(I,x)$ is globally isomorphic to $(G,1)$. In fact, let $u$ be a local section of $I$. Since $I$ is locally isomorphic to $G$, there exists a unique $g_{\varphi}(u)$ in $G$ such that $\varphi(u,u)=u.g_{\varphi}(u)$. It is banal to observe that $g_{\varphi}$ is the required global isomorphism of $(G,H)$-torsors between $(I,x)$ and $(G,1)$. We leave it to the reader to see that upon choosing another local section $u'$ of $I$, we get a global section $s$ of the Saavedra unit $((I,x),\varphi)$. For any two Saavedra units $((I,x),\varphi)$ and $((J,y),\sigma)$ of $\oneG$  with global sections $s$ and $t$, respectively, the morphism
\begin{equation*}\label{the_proposition_eqn_7}
\psi:\xymatrix@1{((I,x),\varphi) \ar[rr] && ((J,y),\sigma)}, 
\end{equation*}
defined by sending $s$ to $t$ provides the isomorphism. The uniqueness of $\psi$ follows from the fact that $s$ and $t$ are uniquely determined by $\varphi$ and $\sigma$, respectively. 
\end{proof}

From the uniqueness of the isomorphism in Proposition \ref{the_proposition}, it follows 
\begin{corollary}\label{corollary:Saavedra_units_form_stack}
$\unitstack(\oneG)$ is a Picard stack.
\end{corollary}

\subsection{The Cocyclic Description of a Saavedra Unit}\label{section:cocycles_saavedra}
Let $((I,x),\varphi)$ be a Saavedra unit of $\oneG$ over $U$ and $V_{\bullet} \ra U$ a hypercover. We assume $\oneG$ is represented by the crossed-module $\lambda:G \ra H$. Let $u \in I(V_0)$ be a local section. A simple cocycle calculation shows that the collection $(g,g_{\varphi}(u),h)$ where $g \in G(V_1)$, $g_{\varphi}(u) \in G(V_0)$, and $h \in H(V_0)$ satisfying the relations
\begin{equation}\label{cocycles_2}
d_2^*(g)d_0^*(g)=d_1^*(g),
\end{equation}
\begin{equation}\label{cocycles_3}
d_0^*(h)=d_1^*(h)\lambda(g),
\end{equation}
\begin{equation}\label{cocycles_7}
gd_0^*(g_{\varphi}(u))=d_1^*(g_{\varphi}(u))d_1^*h^{-1}\lambda(g)d_1^*h\lambda(g),
\end{equation}
\begin{equation}\label{cocycles_8}
\lambda(g_{\varphi}(u))=h,
\end{equation}
represent the Saavedra unit $((I,x),\varphi)$. The collection $(g,g_{\varphi}(u),h)$ is a 1-cocycle with values in the crossed-module
\begin{equation}\label{cocycles_9}
(\id^{-1}_G,\lambda): \xymatrix@1{G \ar[rr] && \ker(\id_H\lambda)}
\end{equation}
defined by $g \mapsto (g^{-1},\lambda(g))$. The action of $\ker(\id_H\lambda)$ on $G$  is $g^{(g',h')}:=g^{h'}$ for any $g \in G$ and $(g',h') \in \ker(\id_H\lambda)$. 

If we chose another local section $u' \in I(V_0)$, we find another 1-cocycle cohomologous to $(g,g_{\varphi}(u'),h)$. Therefore the set of equivalence classes of 1-cocycles with values in the morphism (\ref{cocycles_9}) classify Saavedra units. In fact, the equivalence classes form a group which we denote by $\mathrm{H}^0(*,G\ra \ker(\id_H\lambda))$. Here $*$ represents the final object of the topos of sheaves on $\ES$, i.e. the sheaf whose value is the point at each object of $\ES$. The group structure, induced by the crossed-module structure on (\ref{cocycles_9}), is defined by $(g_1,g'_1,h_1)(g_2,g_2',h_2)=(g_1^{d_0^*h_2}g_2,g_1'^{h_2}g_2',h_1h_2)$. 

\subsection{A Complex of Sheaves defining the Stack of Saavedra Units}\label{section_complex}
Since $\unitstack(\oneG)$ is a contractible Picard stack, it is represented by the class of zero complex $0 \ra 0$. By \cite[Lemme 1.4.13]{Deligne}, any complex quasi-isomorphic to the zero complex provides a representation of $\unitstack(\oneG)$, as well. The following complexes arise naturally in the realm of Saavedra units.
\begin{enumerate}
\item From the cocyclic description of Saavedra units, we know that $\mathrm{H}^0(*,G\ra \ker(\id_H\lambda))$ classify Saavedra units of $\oneG$. We observe that the crossed-module (\ref{cocycles_9}) is equipped with the bracket operation
\begin{equation}\label{bracket}
\{-,-\}:\xymatrix@1{\ker(\id_H\lambda) \times \ker(\id_H\lambda) \ar[rr] && G}
\end{equation}
defined by $\{(g_1,h_1),(g_2,h_2)\} \mapsto g_1g_2g_1^{-1}g_2^{-1}$ which is symmetric, i.e 
\[\{(g_1,h_1)(g_2,h_2)\}\{(g_2,h_2)(g_1,h_1)\}=1,\] 
and Picard, i.e. $\{(g,h)(g,h)\}=1$. Hence, the stack associated to the crossed-module (\ref{cocycles_9}) is Picard and equivalent to $\unitstack(\oneG)$.  We also note that (\ref{cocycles_9}) is the soft truncation of
\begin{equation*}\label{cocycles_10}
\xymatrix@1{G\ar[rr]^(0.42){(\id_G^{-1},\lambda)} && G \ltimes H \ar[rr]^(0.55){(\id_H\lambda)} && H},
\end{equation*}
which is the cone of the morphism $\id_{\oneG}:\oneG \ra \oneG$.
\item From Proposition \ref{the_proposition} and Example \ref{kernel_saavedra_unit}, we deduce that $\ker(\lambda)$ parametrizes Saavedra units of $\oneG$ and their morphisms. Hence, the Picard stack associated to the morphism $\id_{\ker(\lambda)}:\ker(\lambda) \ra \ker(\lambda)$ is equivalent to $\unitstack(\oneG)$.
\end{enumerate}
\begin{remark}
If $\oneA$ is a Picard stack, then $\oneA$ can be represented by a class of abelian sheaf morphisms $\lambda:A \ra B$. In this case, the Saavedra units of $\oneA$ are classified by the abelian group $\mathrm{H}^0(*,A\ra \ker(\id_B\lambda))$. 
\end{remark}
\begin{remark}
We would like to point out the relation between $\unitstack(\oneG)$ the Saavedra units of a (group-like) stack $\oneG$ and $\oneG_1$ the connected components of the identity in $\oneG$. There exists a functor from $\unitstack(\oneG)$ to $\oneG_1$ defined by forgetting the morphism $\varphi:XX \ra X$. $\oneG_1$ has a richer structure which makes it more interesting than $\unitstack(\oneG)$. More details about $\oneG_1$ can be found in \cite{Aldrovandi_Tatar}.
\end{remark}
\subsection{Monoidal Case}\label{monoidal_case}
Let $\oneM$ be a monoidal stack fibered in groupoids over $\ES$. A \emph{Saavedra unit} of $\oneM$ over $U \in \ES$ is a pair $(e,\varphi)$ where $e$ is a cancelable object, i.e. $-\otimes e, e \otimes -:\oneM \ra \oneM$ are fully faithful and $\varphi:ee \ra e$ is an isomorphism. In fact, the functors $-\otimes e, e \otimes -$ are equivalences. For any object $X$, the pre-image of the isomorphism $(\varphi \otimes \id_X)\circ \invass_{e,e,X}$ under the isomorphism $\Hom(eX,X) \ra \Hom(e(eX),eX)$ provides an isomorphism between $X$ and $eX$. We observe that the only difference between a Saavedra unit $(e,\varphi)$ of a monoidal stack and a group-like stack is that in a group-like stack $e$ does not need to be assumed to be cancelable as all objects in a group-like stack are cancelable. From this observation, if $\oneG$ is the substack of cancelable objects of $\oneM$ which is trivially group-like, then the stack $\unitstack(\oneM)$ of Saavedra units of $\oneM$ and the stack $\unitstack(\oneG)$ of Saavedra units of $\oneG$ coincide. Thus, being contractible in $\unitstack(\oneM)$ is same as being contractible in $\unitstack(\oneG)$. Since by Proposition \ref{the_proposition}, $\unitstack(\oneG)$ is contractible, we obtain
\begin{corollary}
$\unitstack(\oneM)$ is contractible groupoid over $\ES$. Hence, it is a Picard stack over $\ES$.
\end{corollary}

This corollary is sheafification of (\cite[Proposition 2.19]{MR2388233}) and proposes a different approach to the same end. One shall note that, in \cite{MR2388233} Kock proves the contractibility of category of Saavedra units by first proving that the category of classical units are contractible then showing that the categories are equivalent. Whereas we only work with Saavedra units without any mention of classical units and we deduce the contractibility of Saavedra units in a monoidal stack from the contractibility of Saavedra units of a group-like stack which we can represent by a crossed-module. 

\section{Weak Units of Group-Like 2-Stacks}\label{section:2stack_of_units}
In this section, we follow the same plan as in section \ref{section:stack_of_units}. We define weak units of a group-like 2-stack $\twoG$ associated to a 2-crossed-module. We call them Joyal-Kock units of $\twoG$. We show that Joyal-Kock units are equivalent to each other up to a unique 2-isomorphism thereby they form a Picard 2-stack $\twounitstack(\twoG)$. We give the cocyclic description of Joyal-Kock units. We compute a complex of sheaves that represent $\twounitstack(\twoG)$. We conclude by discussing the Joyal-Kock units of a monoidal 2-stack.

\subsection{Joyal-Kock Units of a Group-Like 2-Stack}\label{section:joyal_kock_units_in_picard_2_stack}
Let $\twoG$ be a group-like 2-stack represented by the 2-crossed-module 
\begin{equation}\label{2-crossed-module}
\xymatrix@1{G \ar[r]^{\delta} & H \ar[r]^{\lambda} & K},
\end{equation}
of sheaves of groups over $\ES$. We can model $\twoG$ by the group-like 2-stack $\TORS(\oneG,K)$ of $\oneG$-torsors that become trivial over $K$ and where $\oneG \simeq \TORS(G,H)$. However this model has a downside. The commutative square that corresponds to the group-like stack morphism $\oneG \ra K$ that associates to an $(G,H)$-torsor $(P,s)$ a point $\lambda(s)$ in $K$, is not a crossed-square since the the relation $(\left(g^h\right)^k)=\left(g^k\right)^{h^k}$ where $g \in G, h \in H, k \in K$ fails to be satisfied. There exists a better model for $\twoG$. Using the connection between 2-crossed-modules and crossed squares, we model $\twoG$ by $\TORS(\oneL,\oneM)$ where $\Lambda:\oneL \ra \oneM$ is the group-like stack morphism associated to the crossed-square
\begin{equation}\label{diagram:crossed_square}
\begin{tabular}{c}
\xymatrix{L \ar[d]_{u} \ar[r]^{\alpha} \ar@{}[dr]|{\car}&M \ar[d]^v \\ N \ar[r]_{\beta} & P}
\end{tabular}
\end{equation}
obtained from the 2-crossed-module (\ref{2-crossed-module}). 

We refer to  \cite[\S 6.1]{MR1086889} for the definition of a torsor over a group-like stack. Let us remind $\TORS(\oneL,\oneM)$ from \cite[\S6.3.4]{Aldrovandi2009687}.  An object of $\TORS(\oneL,\oneM)$ consists of a pair $(\oneP,x)$, where $\oneP$ is an $\oneL$-torsor and $x:\oneL \ra \oneM$ is an $\oneL$-equivariant map with respect to $\Lambda$. A morphism between any two pairs in $\TORS(\oneL,\oneM)$ is given by the pair $(F,\gamma_F)$
\[(F,\mu_F):\xymatrix@1{(\oneP,x) \ar[rr] && (\oneQ,y)},\]
where $F: \oneP \ra \oneQ$ is an $\oneL$-torsor morphism and $\mu_F$ is a 2-morphism
\begin{equation*}\label{diagram:compatibility_of_torsor_structures}
\xymatrix{\oneP \times \oneL \ar[d] \ar[r]^{F \times \id_{\oneL}} \ar@{}[dr]|{\Da\mu_F}& \oneQ \times \oneL \ar[d] \\ \oneP \ar[r]_{F} & \oneQ}
\end{equation*}
expressing the compatibility of the $\oneL$-torsor structures of $\oneP$ and $\oneQ$. We also require the diagram
\begin{equation*}\label{diagram:commutative_morphism}
\xymatrix{\oneP \ar[rr]^F \ar[dr]_{x} &\ar@{}[d]|{\kesir{\La}{\gamma_F}}& \oneQ \ar[dl]^{y}\\
          &\oneM&}
\end{equation*}
to commute up to a 2-morphism $\gamma_F$.

A 2-morphism in $\TORS(\oneL,\oneM)$
\[\xymatrix@1{**[l](\oneP,x) \ar@/^0.5cm/[rr]^{(F,\mu_F)} \ar@/_0.5cm/[rr]_{(G,\mu_G)} \ar@{}[rr]|{\Da \Gamma}&& **[r](\oneQ,y)},\]
is given by a natural transformation $\Gamma: F \Ra G$ satisfying the equation of natural transformations 
\begin{equation*}
\xymatrix{\oneP \times \oneL \ar@/^0.5cm/[r]^{F \times \id_{\oneL}}  \ar[r]_{G \times \id_{\oneL}}^{\Da \Gamma \times \id} \ar[d] \ar@{}[dr]|{\Da \mu_F}& \oneQ \times \oneL \ar[d] \ar@{}[dr]|{=} & \oneP \times \oneL \ar[r]^{F \times \id_{\oneL}} \ar[d] \ar@{}[dr]|{\Da \mu_G} & \oneQ \times \oneL \ar[d]\\
          \oneP \ar[r]_{G}& \oneQ & \oneP \ar[r]^{F}_{\Da \Gamma} \ar@/_0.5cm/[r]_{G} & \oneQ}
\end{equation*}

\begin{equation*}
\xymatrix{\oneP \ar[rr]_{G}^{\Da \Gamma} \ar@/^0.5cm/[rr]^{F} \ar[dr]_{x}&\ar@{}[d]|{\kesir{\La}{\gamma_G}}& \oneQ \ar[dl]^{y} \ar@{}[dr]|{=}& \oneP \ar[rr]^{F} \ar[dr]_x&\ar@{}[d]|{\kesir{\La}{\gamma_F}}& \oneQ \ar[dl]^y \\&\oneM&&&\oneM&}
\end{equation*}
The tensor product on $\TORS(\oneL,\oneM)$ is similar to the tensor product in the stack case. For the definition of the contracted product of two $\oneL$-torsors, the reader can refer to \cite[\S 6.7]{MR1086889}.

A \emph{Joyal-Kock unit} in $\TORS(\oneL,\oneM)$ is an idempotent $(\oneL,\oneM)$-torsor. That is, an $(\oneL,\oneM)$-torsor $(\oneI,x)$ with an $(\oneL,\oneM)$-torsor morphism
\begin{equation*}
(\varphi,\mu_{\varphi}):\xymatrix@1{(\oneI,x) \otimes (\oneI,x) \ar[rr] && (\oneI,x)},
\end{equation*}
where $\mu_{\varphi}$ is a 2-morphism of the form (\ref{diagram:compatibility_of_torsor_structures}). In other words, with an $(\oneL,\oneM)$-torsor morphism
\begin{equation*}\label{idempotent_2}
(\varphi,\mu_{\varphi}):\xymatrix@1{(\oneI \wedge^{\oneL} \oneI,x^2) \ar[rr]&&(\oneI,x)}.
\end{equation*}
We denote these units in short by $((\oneI,x),\varphi)$.

A \emph{morphism of Joyal-Kock units} in $\TORS(\oneL,\oneM)$
\[\xymatrix@1{((\oneI,x),\varphi) \ar[rr] && ((\oneJ,y),\sigma)},\]
is given by an $(\oneL,\oneM)$-torsor morphism
\[(\psi,\mu_{\psi}):\xymatrix@1{(\oneI,x) \ar[rr] && (\oneJ,y)},\]
and a 2-morphism of $(\oneL,\oneM)$-torsors
\[\xymatrix@R=1cm@C=1cm{(\oneI \wedge^{\oneL} \oneI,x^2) \ar[r]^{\psi \wedge \psi} \ar[d]_{\varphi}& (\oneJ \wedge^{\oneL} \oneJ,y^2)\ar[d]^{\sigma}\\ (\oneI,x) \ar[r]_{\psi} & (\oneJ,y)\ar@{}[ul]|{\Ua\theta_{\psi}}}\]
We denote these morphisms by the pair $(\psi,\theta_{\psi})$.

A \emph{2-morphism of Joyal-Kock units} in $\TORS(\oneL,\oneM)$
\[\xymatrix@1{**[l]((\oneI,x),\varphi) \ar@/^0.5cm/[rr]^{(\psi,\theta_{\psi})} \ar@/_0.5cm/[rr]_{(\phi,\theta_{\phi})} \ar@{}[rr]|{\Da \Gamma}&& **[r]((\oneJ,y),\sigma)},\]
is given by a 2-morphism of $(\oneL,\oneM)$-torsors $\Gamma:\psi \Ra \phi$ satisfying the equation of 2-morphisms
\begin{equation*}\label{equation_of_2_morphisms}
\xymatrix{(\oneI \wedge^{\oneL} \oneI,x^2) \ar@/^0.5cm/[r]^{\psi \wedge \psi} \ar[r]_{\phi \wedge \phi}^{\Da \Gamma \wedge \Gamma} \ar[d]_{\varphi} & (\oneJ \wedge^{\oneL} \oneJ,y^2) \ar[d]^{\sigma} \ar@{}[dr]|{=}&(\oneI \wedge^{\oneL} \oneI,x^2)\ar[r]^{\psi \wedge \psi} \ar[d]_{\varphi}  & (\oneJ \wedge^{\oneL} \oneJ,y^2) \ar[d]^{\sigma}\\
          (\oneI ,x)\ar[r]_{\phi}& (\oneJ,y)  \ar@{}[ul]|{\Ua\theta_{\phi}}&(\oneI,x) \ar[r]^{\psi}_{\Da \Gamma} \ar@/_0.5cm/[r]_{\phi} & (\oneJ,y) \ar@{}[ul]|{\Ua\theta_{\psi}}}
\end{equation*}
This defines the 2-groupoid of Joyal-Kock units of a group-like 2-stack $\twoG$. We denote it by $\twounitstack(\twoG)$. 

\begin{example}\label{kernel_JK_unit}
The trivial $(\oneL,\oneM)$-torsor $(\oneL,1)$ is a Joyal-Kock unit and so is any $(\oneL,\oneM)$-torsor globally equivalent to it.
\end{example}

\subsection{Contractibility of Joyal-Kock Units}
In this section, we show that there exists a Joyal-Kock unit morphism between any two Joyal-Kock units of a group-like 2-stack $\twoG$ modeled by $\TORS(\oneL,\oneM)$ as above.

\begin{proposition}\label{the_proposition_2}
All Joyal-Kock units of $\twoG$ are equivalent up to a unique 2-isomorphism. That is, $\twounitstack(\twoG)$ is a contractible 2-groupoid over $\ES$.
\end{proposition}

\begin{proof}
Let $((\oneI,x),\varphi)$ be a Joyal-Kock unit of $\twoG$ and $\ell$ be a local section of $\oneI$. There exists an $\oneL$-torsor $(P,s)_{\varphi}(\ell)$ unique up to a unique isomorphism satisfying the relation $\varphi(\ell,\ell) \simeq \ell . (P,s)_{\varphi}(\ell)$ from which the non-canonical global equivalence between the $(\oneL,\oneM)$-torsors $(\oneI,x)$ and $(\oneL,1)$ follows. By choosing another local section $\ell'$ in $\oneI$, we get a global section $s$ of the Joyal-Kock unit $((\oneI,x),\varphi)$. We define the morphism between two Joyal-Kock units $((\oneI,x),\varphi)$ and $((\oneJ,y),\sigma)$ by mapping their global sections to each other. This morphism is unique up to the choice of a global section.
\end{proof}

From the uniqueness of the 2-isomorphism in Proposition \ref{the_proposition_2}, it follows that

\begin{corollary}\label{corollary:JK-units_form_2_stack}
$\twounitstack(\twoG)$ equipped with
\[((\oneI,x),\varphi) \otimes ((\oneJ,y),\sigma):=((\oneI \wedge ^{\oneL} \oneJ,xy),\varphi \wedge \sigma),\]
where $\varphi \wedge \sigma$ is of the form (\ref{morphism:additive structure}) is a Picard 2-stack.
\end{corollary}

\subsection{The Cocyclic Description of a Joyal-Kock Unit}
Let $((\oneI,x),\varphi)$ be a Joyal-Kock unit of a group-like 2-stack $\twoG$ modeled by $\TORS(\oneL,\oneM)$ over $U$ and $V_{\bullet} \ra U$ be a hypercover. Upon choosing a local section $\ell \in \oneI(V_0)$ and a simple cocycle calculation shows that the collection $(f,(P,s),g,(Q,t),(K,u))$ where $f \in \oneL(V_2)$, $(P,s) \in \oneL(V_1)$, $g \in \oneL(V_1)$, $(Q,t) \in \oneL(V_0)$, and $(K,u) \in \oneM(V_0)$ satisfying the relations 
\begin{equation}\label{cocycle_6}
d_0^*(u)=d_1^*(u)\beta(s),
\end{equation}
\begin{equation}\label{equation:image_of_sections_4}
g: d_0^*(Q,t)  \simeq  d_1^*(Q,t)(P,s),
\end{equation}
\begin{equation}\label{cocycle_10}
d_2^*(g)d_0^*(g)=fd_1^*(g),
\end{equation}
\begin{equation}\label{cocycle_7}
\beta(t)=u,
\end{equation}
plus the coherence condition on $f$ when it is pulled back to $V_3$ describes a Joyal-Kock unit in $\TORS(\oneL,\oneM)$. We note that this collection is a 1-cocycle with coefficients in the morphism of group-like stacks

\begin{equation}\label{cocycle_8}
(\id^{-1}_{\oneL},\Lambda):\xymatrix@1{\oneL \ar[r] & \TORS(\ker(\id_M\alpha),\ker(\id_P\beta))}.
\end{equation}
We refer to \cite[\S6.1]{MR2387582} for a detailed treatment of cocycles with coefficients in a group-like stack morphism. We shall describe $\TORS(\ker(\id_M\alpha),\ker(\id_P\beta))$. Using the notation in section \ref{section:joyal_kock_units_in_picard_2_stack}, there exists a group-like stack morphism $(\Lambda,\id_{\oneM}):\oneL \rtimes \oneM \ra \oneM$ which corresponds to the commutative square
\begin{equation*}
\xymatrix{**[l]L \rtimes M \ar[r]^{\id_M\alpha} \ar[d]_{(u,v)} & M \ar[d]^{v} \\ **[l]N \rtimes P \ar[r]_{\id_P\beta} & P}
\end{equation*}
where $\oneL \rtimes \oneM$ is the group-like stack associated to the crossed-module $(u,v):L \rtimes M \ra N \rtimes P$ the semi-direct product of the two vertical crossed-modules in diagram (\ref{diagram:crossed_square}). We send the reader to the Appendix for the details of the semi-direct product. $\TORS(\ker(\id_M\alpha),\ker(\id_P\beta))$ is the group-like stack associated to the crossed-module $(u,v):\ker(\id_M\alpha) \ra \ker(\id_P\beta)$. If we choose another local section $\ell' \in \oneI(V_0)$, we find a cohomologous 1-cocycle. Therefore the classes of 1-cocycles with values in the morphism (\ref{cocycle_8}) classify up to equivalence Joyal-Kock units. We denote this set of classes of cocycles by $\mathrm{H}^0(*,\oneL \ra \TORS(\ker(\id_M\alpha),\ker(\id_P\beta))$. This cohomology set makes sense because of the following observation

\begin{proposition}
If $\Lambda:\oneL \ra \oneM$ is a group-like stack morphism that corresponds to a crossed-square, then so is $(\id^{-1}_{\oneL},\Lambda):\oneL \ra \TORS(\ker(\id_M\alpha),\ker(\id_P\beta))$.
\end{proposition}

\begin{proof}
The square corresponding to the group-like stack morphism $(\id^{-1}_{\oneL},\Lambda)$
\begin{equation*}\label{diagram:crossed_square_1}
\xymatrix{L \ar[d]_{u} \ar[r]^{(\id_L^{-1},\alpha)}&**[r]\ker(\id_{M}\alpha) \ar[d]^{(u,v)} \\ N \ar[r]_{(\id_N^{-1}, \beta)} & **[r]\ker(\id_{P}\beta)}
\end{equation*}
endowed with the data 
\begin{itemize}
\item the action of $\ker(\id_P\beta)$ on $\ker(\id_M\alpha)$ is the relation (\ref{semi_direct_product_app}) in the appendix;
\item the action of $\ker(\id_P\beta)$ on $N$ is $n'^{(n,p)}:=n'^{n}$ for every $n,n' \in N$, $p \in P$;
\item the action of $\ker(\id_P\beta)$ on $L$ is $l^{(n,p)}:=l^{u(n)}$ for every $l \in L$, $n \in N$, $p \in P$;
\item the map $\psi:\ker(\id_M\alpha) \times N \ra L$ is defined by $\psi((l,m),n):=\phi((m^{-1})^{\beta(n)},n^{-1})^{m}$ for every $l \in L$, $m \in M$, $n \in N$ and with $\phi:M \times N \ra L$ the map of the crossed square (\ref{crossed_square_app}),
\end{itemize}
is a crossed-square. We leave it to the reader to verify that these data satisfy the conditions CS-1 to CS-9 given in the Appendix.
\end{proof}

\subsection{A Complex of Sheaves defining the 2-Stack of Joyal-Kock Units}
Since $\twounitstack(\twoG)$ is a contractible Picard 2-stack, it is represented by the class of zero complex $0 \ra 0 \ra 0$. By \cite{Aldrovandi_Noohi}, any length 3 complex quasi-isomorphic to the zero complex provides other representations of $\twounitstack(\twoG)$. Here we mention two such complexes that arise naturally in the discussion of Joyal-Kock units that are clearly trivial:
\begin{enumerate}
\item The elements of the set $\mathrm{H}^0(*,\oneL \ra \ker(\id_M\alpha),\ker(\id_P\beta))$ with coefficients in the stack morphism (\ref{cocycle_8}) classify up to equivalence the Joyal-Kock units of $\TORS(\oneL,\oneM)$. From \cite[Proposition 6.2]{MR2387582}, $\mathrm{H}^0(*,\oneL \ra \ker(\id_M\alpha),\ker(\id_P\beta))$ can be identified with set of the 1-cocycles with coefficients in  
\begin{equation}\label{the_complex_2}
\xymatrix@1@C=30pt{L \ar[rr]^(0.35){((\id_L,\alpha^{-1}),u)}  && \ker(\id_M\alpha) \rtimes N \ar[rr]^(0.55){((u,v)(\id_N^{-1},\beta))}&& \ker(\id_P\beta)},
\end{equation}
the cone of (\ref{cocycle_8}). By \cite[Corollary 3.5]{MR2037767}, the complex (\ref{the_complex_2}) is a 2-crossed-module. Moreover it is an exact sequence and hence quasi-isomorphic to the zero complex. The 2-stack associated to it represents the 2-stack of Joyal-Kock units $\twounitstack(\twoG)$.
\item From Proposition \ref{the_proposition_2} and Example \ref{kernel_JK_unit}, we deduce that the identity morphism on $\TORS(\ker(\alpha), \ker(\beta))$ represents $\twounitstack(\twoG)$. Hence, so is its cone
\begin{equation*}\label{the_complex_2_bis}
\xymatrix@1{\ker{(\alpha)} \ar[rr]^(0.4){(\id_{\ker(\alpha)}^{-1},u)}&& \ker(\alpha) \rtimes \ker(\beta) \ar[rr]^(0.6){v\id_{\ker(\beta)}} && \ker(\beta)}.
\end{equation*}
\end{enumerate}
\begin{remark}
If $\twoA$ is a Picard 2-stack, then $\twoA$ can be represented by class of length 3 complex of abelian sheaves $(\delta,\lambda):A \ra B \ra C$. As all the actions are trivial, the commutative square
\begin{equation}\label{diagram:crossed_square_pic}
\begin{tabular}{c}
\xymatrix{A \ar[r] \ar[d]_{\delta} & 0 \ar[d]\\B \ar[r]_{\lambda} & C}
\end{tabular}
\end{equation}
is trivially a crossed-square. Upon repeating the above arguments with (\ref{diagram:crossed_square_pic}), we find that the complex
\begin{equation*}
\xymatrix@1@C=30pt{A \ar[rr]^(0.45){(\id_A,\delta)} && A \oplus B \ar[rr]^(0.42){(\delta,0)+(-\id_B,\lambda)} && \ker(\id_C+\lambda)},
\end{equation*}
represents $\twounitstack(\twoA)$.
\end{remark}

\begin{remark}
We know how to associate a group-like 2-stack to a 2-crossed module, say $(\delta,\lambda):G \ra H \ra K$(see \cite{noohi-2005}). Reciprocally, associating a 2-crossed-module to a group-like 2-stack is also possible whose details will be given in \cite{Aldrovandi_Noohi}. We also remark that it is natural to expect that these results extend in a similar way to $n$-stacks that are at least group-like.
\end{remark}

\subsection{Monoidal Case}\label{2_monoidal_case}
Let $\twoM$ be a monoidal 2-stack fibered in 2-groupoids over $\ES$. A \emph{Joyal-Kock unit} of $\twoM$ over $U \in \ES$ is a pair $(e,\varphi)$ where $e$ is a cancelable object, i.e. $-\otimes e, e \otimes -:\twoM \ra \twoM$ are fully-faithful and $\varphi:ee \ra e$ is a weakly-invertible 1-morphism. As observed in section \ref{monoidal_case}, the fully-faithful 2-functors $-\otimes e, e \otimes -$, that is inducing equivalence on the hom-categories, are in fact biequivalences and if $\twoG$ is the sub-2-stack of cancelable objects of $\twoM$ which is trivially group-like and can be represented by a 2-crossed-module (see \cite{Aldrovandi_Noohi}), then the 2-stack $\unitstack(\twoM)$ of Joyal-Kock units of $\twoM$ and the 2-stack $\unitstack(\twoG)$ of Joyal-Kock units of $\twoG$ coincide. Hence, by Proposition \ref{the_proposition_2}
\begin{corollary}
$\unitstack(\twoM)$ is contractible 2-groupoid over $\ES$. Hence, it is a Picard 2-stack over $\ES$.
\end{corollary}

This Corollary is sheafification of (\cite[Theorem C]{MR2388233}). As in section \ref{monoidal_case}, it proves the contractibility of units of a monoidal 2-stack $\twoM$ without referring to the connection between the classical units and Joyal-Kock units but instead using the complex which represents the group-like 2-stack of cancelable objects of $\twoM$.

\appendix
\section{}\label{Annex}
For the convenience of the reader, we recall some definitions and results related to 2-crossed-modules and crossed squares from \cite{MR2037767}, \cite{MR772056}, and \cite{MR1087375}.

A \emph{2-crossed-module}
\begin{equation*}\label{equation:complex_length_app}
\xymatrix@1{G \ar[r]^{\delta} & H \ar[r]^{\lambda} & K},
\end{equation*}
is a sequence of groups equipped with a right action of $K$ on $H$ and $G$, a right action of $H$ on $G$, and a function $\{~,~\}:H \times H \ra G$ called bracket operation satisfying the axioms
\begin{enumerate}[(2CM-1)]
\item $\delta$ and $\lambda$ are $K$-equivariant where $K$ acts on itself by conjugation;
\item $\delta\{h_0,h_1\}=h_0^{-1}h_1^{-1}h_0h_1^{\lambda h_0}$ for every $h_0,h_1 \in H$;
\item $\{\delta g,h\}=g^{-1}g^h$ for every $g \in G$ and $h \in H$;
\item $\{h,\delta g\}=(g^{-1})^{h}g^{\lambda h}$ for every $g \in G$ and $h \in H$;
\item $\{h_0,h_1h_2\}=\{h_0,h_1\}^{h_0^{-1}h_2h_0}\{h_0,h_2\}$ for every $h_0,h_1,h_2 \in H$;
\item $\{h_0h_1,h_2\}=\{h_1,h_0^{-1}h_2h_0\}\{h_0,h_2\}^{\lambda h_1}$ for every $h_0,h_1,h_2 \in H$;
\item $\{h_0,h_1\}^k=\{h_0^k,h_1^k\}$ for every $h_0,h_1 \in H$ and $k \in K$.
\end{enumerate}

A \emph{crossed-square} is a commutative diagram of groups and group homomorphisms
\begin{equation}\label{crossed_square_app}
\begin{tabular}{c}
\xymatrix{L \ar[r]^{\alpha} \ar[d]_{u} \ar@{}[dr]|{\car}& M \ar[d]^v\\N \ar[r]_{\beta}&P}
\end{tabular}
\end{equation}
equipped with a right action of $P$ on $L$, on $M$, and on $N$, a function $\phi:M \times N \ra L$ satisfying the axioms
\begin{enumerate}[(CS-1)]
\item $u$ and $\alpha$ are $P$-equivariant;
\item $v$, $\beta$, $v \circ \alpha$, and $\beta \circ u$ are crossed-modules;
\item $\alpha \circ \phi (m,n)=(m^{-1})^{\beta(n)}m$ for every $m \in M$, $n \in N$;
\item $u \circ \phi(m,n)=n^{-1}n^{v(m)}$ for every $m \in M$, $n \in N$;
\item $\phi(\alpha(l),n)=(l^{-1})^{\beta(n)}$l for every $l \in L$, $n \in N$;
\item $\phi(m,u(l))=l^{-1}l^{v(m)}$ for every $l \in L$, $m \in M$;
\item $\phi(m_0m_1,n)=\phi(m_1,n)\phi(m_0,n)^{v(m_1)}$ for every $m_0,m_1 \in M$, $n \in N$;
\item $\phi(m,n_0n_1)=\phi(m,n_0)^{\beta(n_1)}h(m,n_1)$ for every $m \in M$, $n_0,n_1 \in N$;
\item $\phi(m,n)^p=\phi(m^p,n^p)$.
\end{enumerate}
It follows from the above axioms that the homomorphisms $u$ and $\alpha$ are crossed-modules, as well.
\begin{corollary*}(\cite[Corollary 3.5]{MR2037767})
The cone of the crossed-square (\ref{crossed_square_app}) 
\begin{equation*}\label{2-crossed-module_app}
\xymatrix@1{L \ar[r]^(0.4){\partial_2} & M \rtimes N \ar[r]^(0.6){\partial_1} & P,}
\end{equation*}
is a 2-crossed-module where
\begin{itemize}
\item $\partial_1(m,n)=\beta(n)v(m)$ and $\partial_2(l)=(\alpha(l)^{-1},u(l))$ for every $m \in M$, $n \in N$, $l \in L$;
\item the action of $P$ on $M \rtimes N$ is $(m,n)^p:=(m^p,n^p)$ for every $m \in M$, $n \in N$, $p \in P$;
\item the action of $P$ on $L$ is the one that comes with the crossed square (\ref{crossed_square_app});
\item the action of $M \rtimes N$ on $L$ is $l^{(m,n)}:=l^{\beta(n)}$ for every $l \in L$, $n \in N$, $m \in M$;
\item the bracket operation $\{~,~\}:(M \rtimes N) \times (M \rtimes N) \ra L$ is given by $\{(m_1,n_1),(m_2,n_2)\}:=\phi(m_1,n_1^{-1}n_2n_1)\}$ for every $m_i \in M$, $n_i \in N$.
\end{itemize}
\end{corollary*}
Let $u:L \ra N$ and $v: M \ra P$ be the two vertical crossed-modules in the crossed-square (\ref{crossed_square_app}). The \emph{semi-direct product of $u$ and $v$}  is the crossed-module $(u,v):L \rtimes M \ra N \rtimes P$ where the action of $N \rtimes P$ on $L \rtimes M$ is given by 
\begin{equation}\label{semi_direct_product_app}
(l,m)^{(n,p)}:=(\phi(m^p,n)^{-1}l^{p\beta(n)},m^p).
\end{equation}
This definition of semi-direct product is deduced form the definition of the semi-direct product of any two crossed-modules where one acts on the other. We refer to \cite{MR1087375} for the details of actions of crossed-modules.

\bibliographystyle{plain}

\end{document}